\documentclass[12pt]{amsart}

\usepackage{color}
\definecolor{mred}{rgb}{0.7,0,0}
\definecolor{mblue}{RGB}{034,113,179}
\usepackage[colorlinks=true,citecolor=mblue,linkcolor=mred,backref=page]{hyperref}
\usepackage{enumitem}
\usepackage[letterpaper,total={6in, 8.5in},top=1.3in,asymmetric]{geometry}
\usepackage[normalem]{ulem}

\numberwithin{equation}{section}
\usepackage[colorinlistoftodos]{todonotes}

\usepackage{amssymb}
\usepackage{amsmath}
\usepackage{amsfonts}
\usepackage{amsthm}
\usepackage{mathrsfs}

\newtheorem{Theorem}{Theorem}[section]
\newtheorem{Lemma}[Theorem]{Lemma}

\newtheorem{Proposition}[Theorem]{Proposition}

\theoremstyle{definition}

\theoremstyle{remark}
\newtheorem{Remark}[Theorem]{Remark}

\def \RE{\text{Re}}

\def \T{{\mathcal T}}

\def \0{\lambda_{0}}

\def \A{{\mathbb A}}

\newcommand{\Id}{\mathrm{id}}
\newcommand{\vol}{\mathrm{vol}}
\newcommand{\conf}{\mathrm{conf}}
\newcommand{\Teich}{\mathrm{Teich}}
\newcommand{\mc}{\mathcal}

\DeclareMathOperator{\im}{Im}


\begin{document}
\title[]{Quasi-fuchsian flows and the coupled vortex equations}

\author[M. Ceki\'{c}]{Mihajlo Ceki\'{c}}
\date{\today}
\address{Institut f\"ur Mathematik, Universit\"at Z\"urich, Winterthurerstrasse 190,
CH-8057 Z\"urich, Switzerland}
\curraddr{CNRS \&  LAMA, Universit\'e Paris-Est Cr\'eteil, 61 Av. du Général de Gaulle Centre, 94000 Créteil, France}
\email{mihajlo.cekic@cnrs.fr}

\author[G.P. ~Paternain]{Gabriel P. Paternain}
\address{ Department of Mathematics, University of Washington, Seattle, WA 98195, USA}
\email {gpp24@uw.edu}


\date{\today}

\begin{abstract} We provide an alternative construction of the quasi-Fuchsian flows introduced by Ghys in  \cite{Ghys-92}.  Our approach is based
on the coupled vortex equations that allows to see these flows as thermostats on the unit tangent bundle of the Blaschke metric
uniquely determined by a conformal class and a holomorphic quadratic differential. We also give formulas for the marked length spectrum of a quasi-Fuchsian flow in the thermostat parametrization.
\end{abstract}
 
\maketitle

\section{Introduction}

In 1992, Ghys introduced a remarkable class of Anosov flows. Given a closed oriented surface $M$ of genus $\geq 2$, let $g_1$ and $g_2$ be two metrics of constant curvature $-1$, and let $[g_{1}]_{\Teich}$ and $[g_{2}]_{\Teich}$ be the corresponding points in Teichm\"uller space $\T(M)$. Th\'eor\`eme B in \cite{Ghys-92} establishes the existence of a smooth, orientable Anosov foliation $\phi_{[g_{1}]_{\Teich}, [g_{2}]_{\Teich}}$ on the unit tangent bundle $T_{1}M$ of positive half-lines tangent to the surface $M$, with the following properties: any flow parametrizing $\phi_{[g_{1}]_{\Teich}, [g_{2}]_{\Teich}}$ has a smooth weak stable foliation that is $C^{\infty}$-conjugate to the weak stable foliation of $g_1$, and a smooth weak unstable foliation that is $C^{\infty}$-conjugate to the weak unstable foliation of $g_2$. Moreover, a volume form is preserved if and only if $[g_{1}]_{\Teich} = [g_{2}]_{\Teich}$. Ghys called these flows \emph{quasi-Fuchsian flows}, as his construction is analogous to that of quasi-Fuchsian groups obtained by coupling two Fuchsian groups \cite[IX.F]{Maskit_88}.

In the realm of Anosov flows on closed 3-manifolds, quasi-Fuchsian flows hold a distinguished position due to the following classification result, also due to Ghys: \cite[Th\'eor\`eme 4.6]{Ghys_93} asserts that any Anosov flow on a 3-manifold with smooth weak stable and unstable foliations must be $C^{\infty}$-orbit equivalent (up to finite covers) to either a quasi-Fuchsian flow or the suspension of a diffeomorphism of the 2-torus.

One can develop intuition for quasi-Fuchsian flows by first considering the case where $g_1$ and $g_2$ are $C^3$-close. In this case, when considering their respective geodesic flows on $T_{1}M$, the weak stable foliation of $g_1$ is transverse to the weak unstable foliation of $g_2$. Their intersection defines the Anosov foliation $\phi_{[g_{1}]_{\Teich}, [g_{2}]_{\Teich}}$, and hence \cite[Th\'eor\`eme B]{Ghys-92} can be viewed as a statement about extending this class of flows for \emph{arbitrary pairs} of points in Teichm\"uller space. Ghys' construction is based on introducing a new differentiable structure on $T_{1}M$ such that the geodesic flow of a fixed background hyperbolic metric $h$ admits a transverse $(\text{PSL}(2,\mathbb{R}) \times \text{PSL}(2,\mathbb{R}), \mathbb{RP}^{1} \times \mathbb{RP}^{1})$-structure, with a holonomy naturally constructed by coupling the holonomies of $g_1$ and $g_2$ (see \cite[Section 1]{BG_84} and \cite[1.2.C]{CC_00} for background on transverse structures). Then, using the fact that homeomorphic closed 3-manifolds are diffeomorphic, $\phi_{[g_{1}]_{\Teich}, [g_{2}]_{\Teich}}$ is obtained by pulling back the geodesic flow of $h$ via this abstract (i.e., not explicit) diffeomorphism. While this construction is elegant and well-suited for classification, it complicates certain computations, especially if one is interested—as we are—in explicitly computing resonant 1-forms, winding cycles of SRB measures, and helicity, as in \cite{CP22} (for background on the microlocal approach to hyperbolic flows see \cite{Lefeuvre-book}).

In \cite{Paternain-07}, the second author pointed out that quasi-Fuchsian flows may also arise explicitly as thermostat flows associated with a system of PDEs known as the \emph{coupled vortex equations}, and raised the question of whether this construction describes the same class introduced by Ghys. In this paper, we confirm that this is indeed the case. To set the context, it is useful to describe a broader class of Anosov thermostat flows associated with data given by a closed Riemann surface and a holomorphic differential of degree $m$ on the surface.

Let $(M, g)$ be a closed, oriented Riemannian surface of genus $\geq 2$, and let $\pi:SM\to M$ be its unit circle bundle. Denote by $X$ the geodesic vector field and by $V$ the generator of (oriented) rotations in the fibers of $SM$. Recall that the canonical line bundle $\mc{K} := (T_{\mathbb{C}}^*M)^{1, 0}$ is the holomorphic part of the complexified cotangent bundle $T_{\mathbb{C}}^*M$. For $m \in \mathbb{Z}_{\geq 0}$, there is a natural map
\[
\pi_m^*: C^\infty(M; \mc{K}^{\otimes m}) \to C^\infty(SM), \quad \pi_m^*T(x, v):= T_x(v, v, \dotsc, v).
\]
If $A \in C^\infty(M; \mc{K}^{\otimes m})$, we are interested in the flow generated by the vector field 
\begin{equation}\label{eq:thermostat-formula}
	F := X + \lambda V, \quad \lambda = \im(\pi_m^*A).
\end{equation}
The pair $(g, A)$ is said to satisfy the \emph{coupled vortex equations} if:
\begin{equation}\label{eq:coupled-vortex-intro}
	\bar{\partial} A = 0, \quad K_g = -1 + (m - 1)|A|_{g}^2,
\end{equation}
where $K_{g}$ is the Gaussian curvature of $g$, and $\bar{\partial}$ is the holomorphic derivative with respect to the complex structure defined by $g$. If the system \eqref{eq:coupled-vortex-intro} is satisfied, \cite[Theorem 5.1]{Mettler-Paternain-19} shows that $F$ is Anosov, 
and \cite[Theorem 5.5]{Mettler-Paternain-19} shows $F$ preserves an absolutely continuous measure if and only if $A = 0$.

By \cite[Remark 5.3]{Mettler-Paternain-19}, for any holomorphic differential $A$ of degree $m$, there is a unique metric $g$ in the conformal class $[g]_{\conf}$ that solves \eqref{eq:coupled-vortex-intro}, which we will refer to as the \emph{Blaschke metric}. The unique hyperbolic metric in this class will be denoted by $\sigma$, so that $[g]_{\conf} = [\sigma]_{\conf}$, but unless $A = 0$, $g$ and $\sigma$ do \emph{not} coincide. Hence, this class of flows is parametrized by the data $([g]_{\conf}, A)$, i.e. by a Riemann surface structure on $M$ and a holomorphic differential of degree $m$. The group $\text{Diff}_{0}(M)$ of diffeomorphisms isotopic to the identity acts naturally on this data and we shall denote the quotient by $\T\mathrm{D}_m(M)$, the Teichm\"uller space of holomorphic differentials of degree $m$. For $m=2$ we can think of $\T\mathrm{D}_2(M)$ as the tangent (or cotangent) bundle of $\T(M)$, and in this case we will write $\T\mathrm{QD}(M)$ for $\T\mathrm{D}_2(M)$.

This paper is concerned with the case $m=2$, where it is known that the weak bundles are $C^{\infty}$ and are explicitly described as follows (see \cite{Paternain-07} and \cite[Section 6.1]{Mettler-Paternain-19}):
\begin{equation}\label{eq:weakbundles}
\mathbb{R}F \oplus \mathbb{R} (H + r^{s,u}V),
\end{equation}
where $H := [V,X]$ is the horizontal vector field and
\begin{equation}\label{eq:rus}
	r^{u} = 1 + \frac{V\lambda}{2}, \quad r^{s} = -1 + \frac{V\lambda}{2}.
\end{equation}
In order to show that the class of thermostat flows defined by \eqref{eq:coupled-vortex-intro} for $m=2$ provides an alternative description of the class of quasi-Fuchsian flows introduced by Ghys, we need to exhibit a \emph{bijection}
\[\mc{G}: \T\mathrm{QD}(M) \to \T(M) \times \T(M)\]
such that if we let $([g_{1}]_{\Teich}, [g_{2}]_{\Teich}) := \mc{G}([g]_{\conf},A)$, then the weak stable foliation of the hyperbolic metric $g_1$ is smoothly conjugate to the foliation defined by \eqref{eq:weakbundles} for $r^s$, and the weak unstable foliation of the hyperbolic metric $g_2$ is smoothly conjugate to the foliation defined by \eqref{eq:weakbundles} for $r^u$. The bijection $\mc{G}$ that we will use is well known in the literature from different points of view. In this note we choose
the harmonic map approach to Teichmüller theory (see \cite{EE69,Tromba-92}) but we could have also used the Anti-de Sitter (AdS) point of view as originated in the paper by Mess \cite{Mess_07}; we briefly survey this alternative approach in Remark \ref{remark:ads} below.

Given two hyperbolic metrics $\sigma$ and $\rho$ on $M$, it is well-known that there exists a unique harmonic diffeomorphism $f: (M,\sigma) \to (M,\rho)$ isotopic to the identity. Moreover, the \emph{Hopf quadratic differential} $(f^*\rho)^{2,0}$ is holomorphic (here the $(2, 0)$-part of $f^*\rho$ is taken with respect to the complex structure derived from $[\sigma]_{\conf}$), and thus we have a map
\[\Phi_{[\sigma]_{\conf}}: \T(M) \to  \text{QD}([\sigma]_{\conf}),\;\;\; [\rho]_{\Teich} \mapsto (f^*\rho)^{2,0},\]
where $\text{QD}([\sigma]_{\conf})$ denotes the vector space of holomorphic quadratic differentials on the Riemann surface determined by $[\sigma]_{\conf}$. This map is studied in \cite{Samp78,Wolf89}, where it is shown that it is a \emph{homeomorphism}, see \cite[Theorem 3.1]{Wolf89}. (Note that this map only depends on the conformal class $[\sigma]_{\conf}$, which is being used as a \emph{base point}.)

Let us introduce some notation concerning the \emph{marked length spectrum}. For a non trivial free homotopy class $\mathfrak{c}$ in $M$ and a hyperbolic metric $\sigma$, write $\ell_\sigma(\mathfrak{c})$ for the length (with respect to $\sigma$) of the unique closed geodesic in the class $\mathfrak{c}$.

We are now ready to state our main result:

\begin{Theorem}\label{thm:main}
The map $\mc{G}$ defined by
\[\mc{G}([\sigma]_{\conf}, A) = (\Phi_{[\sigma]_{\conf}}^{-1}(A), \Phi_{[\sigma]_{\conf}}^{-1}(-A)),\]
where $[\sigma]_{\conf}$ is a conformal class on $M$ and $A \in \mathrm{QD}([\sigma]_{\conf})$, descends to a bijection $\mc{G}: \T \mathrm{QD}(M)\to \T(M)\times \T(M)$ such that if we define $([g_{1}]_{\Teich}, [g_{2}]_{\Teich}) := \mc{G}([\sigma]_{\conf}, A)$ and we let $g\in [\sigma]_{\conf}$ be the Blaschke metric, then:
\begin{enumerate}[topsep=4pt, itemsep=4pt, label=\alph*.]
    \item The weak stable foliation of the geodesic flow of the hyperbolic metric $g_1$ is smoothly conjugate to the foliation tangent to \eqref{eq:weakbundles} for $r^s$.
    \item The weak unstable foliation of the geodesic flow of the hyperbolic metric $g_2$ is smoothly conjugate to the foliation tangent to \eqref{eq:weakbundles} for $r^u$.
    \item The conjugacies between the foliations from the previous two items are explicit and are given in Section \ref{section:I} below in terms of the holomorphic quadratic differential $A$ and the Blaschke metric $g$ determined by the pair $([\sigma]_{\conf}, A)$.
    \item The flow of $F$ preserves a volume form if and only if $[g_{1}]_{\Teich} = [g_{2}]_{\Teich}$, which holds if and only if $A = 0$.
    \item For any non trivial free homotopy class $\mathfrak{c}$ in $M$, there is a unique closed orbit $\zeta$ of the flow generated by \eqref{eq:thermostat-formula} such that the projection $\pi \circ \zeta$ of $\zeta$ to $M$ belongs to $\mathfrak{c}$, and
\begin{equation}\label{eq:arithmetic-mean}
	\ell_g(\pi \circ \zeta) = \frac{1}{2}(\ell_{g_1}(\mathfrak{c}) + \ell_{g_2}(\mathfrak{c})),
\end{equation}
where $\ell_g(\pi \circ \zeta)$ is the length of $\pi \circ \zeta$ with respect to $g$.
\end{enumerate}
\end{Theorem}

Theorem \ref{thm:main}, Items a--d, provide an independent ``PDE proof'' of \cite[Théorème B]{Ghys-92}, and the classification result \cite[Théorème 4.6]{Ghys_93} shows that the two constructions are equivalent up to smooth orbit equivalence. We note that the appealing formula in Item e, expressing the marked length spectrum of the thermostat flow as the arithmetic mean of those of $g_1$ and $g_2$, depends on the particular parametrization of the flow and is a novel feature not detected in \cite{Ghys-92} (note that $\ell_g(\pi\circ \zeta)$ coincides with the period of $\zeta$). It is interesting to note that Ghys' proof of Item d relies on the marked length spectrum rigidity for hyperbolic metrics, while in our case we invoke \cite[Theorem 5.5]{Mettler-Paternain-19}, which relies on an energy identity and ideas from geometric inverse problems. One possible approach to show that $\mc{G}$ is a bijection, is to consider a functional given by the sum of two Dirichlet energies and use the argument in the proof of \cite[Proposition 2.12]{Schoen93} to establish that this functional has a unique critical point (a global minimum).

In summary, here is how our construction of quasi-Fuchsian flows proceeds, starting from a pair of points $([g_{1}]_{\Teich}, [g_{2}]_{\Teich})$:
\begin{itemize}
    \item The sum of Dirichlet energies associated with the pair $([g_{1}]_{\Teich}, [g_{2}]_{\Teich})$ produces a unique global minimum on $\T(M)$ and hence a point $[\sigma]_{\Teich} \in \T(M)$, which, as explained in \cite{Schoen93}, should be thought of as the midpoint in $\T(M)$ between $[g_{1}]_{\Teich}$ and $[g_{2}]_{\Teich}$ since the pair of harmonic maps thus obtained have opposite Hopf differentials, $A$ and $-A$.
    \item With the pair $([\sigma]_{\conf}, A)$ obtained above, we solve the coupled vortex equations \eqref{eq:coupled-vortex-intro} to obtain a unique Blaschke metric $g \in [\sigma]_{\conf}$.
    \item On the unit circle bundle $SM_g$ of $g$, we consider the thermostat flow of $F = X + \lambda V$, where $\lambda = \im(\pi_2^*A)$.
    \item It turns out that the flow of $F$ is Anosov, with the smooth weak stable (resp. unstable) foliation conjugate to the weak stable (resp. unstable) foliation of the geodesic flow of $g_{1}$ (resp. $g_{2}$).
\end{itemize}

Reversing the process, starting with a pair $(g, A)$ solving the coupled vortex equations \eqref{eq:coupled-vortex-intro}, we can determine the pair $([g_{1}]_{\Teich}, [g_{2}]_{\Teich})$ by giving explicit hyperbolic representatives in the two classes:
\[2\text{\rm Re}\, A+\left(1+|A|^{2}_{g}\right)g\in [g_{1}]_{\Teich}\] 
and
\[-2\text{\rm Re}\, A+\left(1+|A|^{2}_{g}\right)g\in [g_{2}]_{\Teich}.\]
Often it is more practical to describe the points $[g_{1}]_{\Teich}$ and $[g_{2}]_{\Teich}$ using the marked length spectrum. In the course of proving the main theorem, we will also show that for a free homotopy class $\mathfrak{c}$ in $M$ we have
\begin{equation}\label{eq:mls1}
 \ell_{g_{1}}(\mathfrak{c}) = \ell_{g}(\pi \circ \zeta) + \frac{1}{2}\int_{\zeta}V\lambda,
\end{equation}
where $\zeta$ is as in Item e. Similarly, the marked length spectrum for $g_2$ is given by
\begin{equation}\label{eq:mls2}
\ell_{g_{2}}(\mathfrak{c}) = \ell_{g}(\pi \circ \zeta) -\frac{1}{2}\int_{\zeta}V\lambda.
\end{equation}
Item e from Theorem \ref{thm:main} follows right away from \eqref{eq:mls1} and \eqref{eq:mls2}.

An amusing consequence of these identities is a ``new'' proof of the marked length spectrum rigidity for hyperbolic metrics. Indeed, if $\ell_{g_{1}}(\mathfrak{c}) = \ell_{g_{2}}(\mathfrak{c})$ for all $\mathfrak{c}$, then
\[\int_{\zeta}V\lambda = 0\]
for all closed orbits $\zeta$ of $F$. By the Livšic theorem, $V\lambda$ is a coboundary, and hence $F$ preserves a volume form (since $V\lambda$ is the divergence of $F$). By Theorem \ref{thm:main}, Item d, this implies that $A = 0$, and thus $[g_{1}]_{\Teich} = [g_{2}]_{\Teich}$.
\medskip

In addition to the applications in \cite{CP22} mentioned above, we may raise also a few interesting questions related to our approach. Namely, quasi-Fuchsian flows associate an Anosov flow on a $3$-manifold to a pair of co-compact representations $\pi_1(M) \to \mathrm{PSL}(2, \mathbb{R})$; in a parallel development, from the same data \cite{Delarue-Guillarmou-Monclair-25} (see also \cite{Bonthonneau-Lefeuvre-Weich-26}) construct an Axiom A flow from the spacelike geodesic flow on the $3$-dimensional Anti-de Sitter space, with the same periods of closed orbits. It is natural to ask what is the precise relation between these two flows. Another question is to clarify semisimplicity for resonant forms and more generally the behaviour of the Ruelle zeta function for quasi-Fuchsian flows \cite{CP22}; we do this for cubic differentials in \cite{Cekic-Paternain-26}.

\medskip
This paper is organized as follows: Section \ref{section:preliminaries} provides preliminaries on the geometry of the unit circle bundle, the coupled vortex equations, harmonic maps, and the Hopf differential. Section \ref{section:G} explains why the map $\mathcal G$ is a bijection. Section \ref{section:I} introduces the explicit conjugacy between foliations and proves Theorem \ref{thm:main}. The final Section \ref{section:mls} provides proofs of the marked length spectrum identities \eqref{eq:mls1} and \eqref{eq:mls2}.

\subsection*{Acknowledgements}  We are grateful to Martin Leguil and Andrey Gogolev for discussions related to their recent work \cite{GLRH_24}. GPP was supported by NSF grant DMS-2347868. MC received funding from an Ambizione grant (project number 201806) from the Swiss
National Science Foundation.

\section{Preliminaries} \label{section:preliminaries}

\subsection{Geometry of the unit circle bundle}\label{ssec:geometry} We refer to \cite{Guillemin-Kazhdan-80, Merry-Paternain-11, Singer-Thorpe-76} for details on the contents of this subsection. Let $(M, g)$ be an oriented closed Riemannian surface and let $X$ be the geodesic vector field on the unit sphere bundle $SM = SM_g$. Denote by $V$ the generator of the fibrewise rotation action and by $H$ the horizontal vector field. Let $K_g$ be the Gauss curvature of $(M, g)$. Then:
\begin{align}\label{eq:surface-geometry}
\begin{split}
	[H, V] &= X,\\
	[V, X] &= H,\\
	[X, H] &= K_{g}V.
\end{split}
\end{align}
The frame $\{X, H, V\}$ has a dual global coframe of $1$-forms $\{\alpha, \beta, \psi\}$ that satisfies the structure equations:
 \begin{align}\label{eq:surface-geometry2}
\begin{split}
	d\alpha &= \psi\wedge\beta,\\
	d\beta &= -\psi\wedge\alpha,\\
	d\psi&= -K_{g}\,\alpha\wedge\beta.
\end{split}
\end{align}
For $k \in \mathbb{Z}$ and $x \in M$, introduce
\[\mho_k(x) = \{f \in C^\infty(S_xM) \mid V f = ik f\}.\]
It is straightforward to see that $\mho_k \to M$ has the structure of a smooth line bundle; a smooth section of this bundle may be seen as a smooth function on $SM$. In fact if $\mc{K} := (T_{\mathbb{C}}^*M)^{1, 0}$ is the associated \emph{canonical} bundle and $\mc{K}^{-1} := (T_{\mathbb{C}}^*M)^{0, 1}$, then for any $k \in \mathbb{Z}$, $\mho_k$ may be identified with the tensor power $\mc{K}^{\otimes k}$ (here, if $k < 0$ we set $\mc{K}^{\otimes k} = (\mc{K}^{-1})^{\otimes |k|}$). For any $k \in \mathbb{Z}_{\geq 0}$ and $x \in M$, there are natural maps
\[\pi_{k}^*: (\otimes_S^k T_{\mathbb{C}}^*M)_x \to C^\infty(S_xM), \quad \pi_k^*T(v) = T(v, \dotsc , v),\]
where $\otimes_S^k T_{\mathbb{C}}^*M$ denotes the bundle of symmetric tensors of degree $k$. Denote by $\otimes_S^k T_{\mathbb{C}}^*M|_{0-\mathrm{tr}}$ the sub-bundle of \emph{trace-free} symmetric tensors, where a symmetric $k$-tensor $T$ at $x$ is trace free if for an orthonormal basis $(\mathbf{e}_i)_{i = 1}^2$ of $T_xM$, $T(\mathbf{e}_1, \mathbf{e}_1, \dotso) + T(\mathbf{e}_2, \mathbf{e}_2, \dotso) = 0$. It is straightforward to see that there are vector bundle isomorphisms
\[\forall k \in \mathbb{Z}_{> 0}, \quad \otimes_S^k T_{\mathbb{C}}^*M|_{0-\mathrm{tr}} = \mc{K}^{\otimes k} \oplus \mc{K}^{\otimes (-k)}; \quad \forall k \in \mathbb{Z}, \quad \pi_{|k|}^*: \mc{K}^{\otimes k} \xrightarrow{\cong} \mho_k.\]
Then $\pi_{|k|}^*$ identifies smooth sections of $\mc{K}^{\otimes k}$ with $C^{\infty}(M,\mho_{k})$.

\subsection{Coupled vortex equations} Let $(M, g)$ be a closed oriented Riemannian surface of genus $\geq 2$. As pointed out in the introduction an interesting class of thermostats is obtained when $\lambda$ is generated by a differential $A$ of degree $m\geq 2$ (that is, a smooth section of $\mc{K}^{\otimes m}$) such that $g$ and $A$ are linked by the \emph{coupled vortex equations}
\begin{equation}\label{eq:coupled-vortex-equations}
	\bar{\partial} A = 0, \quad K_g = -1 + (m - 1)|A|_{g}^2.
\end{equation}
(See  \cite{Li_19} for the relation with the classical vortex equations in gauge theory.)
It is known that when \eqref{eq:coupled-vortex-equations} holds, the flow of $F:=X+\lambda V$ for 
\begin{equation}\label{eq:lambda-A}
	\lambda := \im(\pi_m^*A)
\end{equation}
is Anosov, see \cite[Theorem 5.1]{Mettler-Paternain-19}. We note that by \cite[Lemma 5.2]{Mettler-Paternain-19}, we have 
\begin{equation}\label{eq:curvature-condition-coupled-vortex}
	-1 \leq K_g < 0
	\end{equation}
	and by \cite[Remark 5.3]{Mettler-Paternain-19}, for every holomorphic differential $A$ on $(M, [g]_{\conf})$ we obtain a unique solution $(g,A)$ to the coupled vortex equations \eqref{eq:coupled-vortex-equations}. We refer to $g$ as the {\it Blaschke metric} given that for $m=3$ this matches precisely the metric arising in the context of affine spheres
	\cite{Wang_91}.
	
	Later on it will be useful for us to pull-back $A$ to $SM$ via the canonical projection $\pi:SM\to M$. By \cite[page 563 and (4.2)]{Mettler-Paternain-19} (alternatively, see \cite[Appendix A]{CP22}), we may write
	\[\pi^*A=((V\lambda)/m+i\lambda)(\alpha+i\beta)^{\otimes m},\]
	and we have
	\begin{equation}
	\pi^*|A|_{g}^2=(V\lambda)^2/m^2+\lambda^2,
	\label{eq:norm}
	\end{equation}
	and for $m=2$ in particular it holds that
\begin{equation}
\RE\,\pi^*A=	 \frac{1}{2}V\lambda(\alpha \otimes \alpha - \beta \otimes \beta) - \lambda(\alpha \otimes \beta + \beta \otimes \alpha).
\label{eq:reA}
\end{equation}
Finally $A$ being holomorphic translates into \cite[Lemma 4.1]{Mettler-Paternain-19} (alternatively see \cite[Lemma A.2]{CP22}):
\begin{equation}
\label{eq:Ahol}
XV\lambda-mH\lambda=0,\quad HV\lambda+mX\lambda=0,
\end{equation}
where we note that the second equality follows from the first one thanks to the structure equations \eqref{eq:surface-geometry}.

\subsection{Harmonic maps and the Hopf differential}\label{ssec:harmonic-map} See \cite[Chapter 9]{Jost-17}, as well as  \cite{EE69,Tromba-92, Farb-Margalit-12} for background and more details about this section. Given a smooth map 
$$f:(M,\sigma)\to (N,\rho)$$
between closed Riemannian manifolds, we define the \emph{Dirichlet energy} of $f$ as
\[E(f):=\int_{M}e(f)\, d\text{vol}_{\sigma},\]
where $d\text{vol}_{\sigma}$ denotes the volume form of $\sigma$ and
\[e(f):=\frac{1}{2}\langle df, df\rangle_{T^{*}M\otimes f^{*}TN}=\frac{1}{2}\text{Tr}_{\sigma}(f^*\rho),\]
where $\langle{\bullet, \bullet}\rangle_{T^*M \otimes f^*TN}$ is the fibrewise inner product defined from $\sigma, \rho$, and $f$, and $\mathrm{Tr}_{\sigma}(\bullet)$ denotes the trace of a $2$-tensor $\bullet$ with respect to $\sigma$. The Euler-Lagrange equation for $E$ is the harmonic map equation $\text{Tr}_{\sigma}(\nabla df)=0$, where $\nabla$ is the tensor product of Levi-Civita connections on $T^*M \otimes f^*TN$. Clearly $E$ (and $e$) depend on $\sigma$ and $\rho$; when we would like to emphasise this dependence we will write $E_{\sigma, \rho}$.

When $M$ is two-dimensional (and orientable), it is easy to see from the definition that $e(f)\, d\vol_\sigma$ as well as the Dirichlet energy only depend on the conformal class $[\sigma]_{\conf}$, and hence the harmonic map equation only depends on the complex structure on $M$. The quadratic differential $A:=(f^{*}\rho)^{2,0}$ is called the {\it Hopf differential} of $f$. As an important fact, a direct computation in local (isothermal) coordinates shows that if $f$ is harmonic, then $A$ is holomorphic, see \cite[pages 451 and 452]{Wolf89} (and references therein). (Conversely, if $A$ is holomorphic, then $f$ is harmonic assuming also that the Jacobian of $f$ is non-zero.) In this paper, we are exclusively concerned with the case when the target $N=M$ and $\rho$ is a hyperbolic metric. Then, given the data $[\sigma]_{\conf}$ and $\rho$ on $M$, there exists a unique harmonic diffeomorphism $f_{\sigma,\rho}:=f:(M,[\sigma]_{\conf})\to (M,\rho)$ such that $f$ is isotopic to the identity (see \cite[Theorem 9.7.2]{Jost-17}). 

Recall that Teichm\"uller space $\T(M)$ is given by equivalence classes of hyperbolic metrics on $M$ under the action of the group $\text{Diff}_{0}(M)$ of diffeomorphisms isotopic to the identity. Given $\varphi, \phi\in \text{Diff}_{0}(M)$ a straightforward computation from the defining Euler-Lagrange equations shows that
\begin{equation}\label{eq:equivariance}
	f_{\varphi^*\sigma,\phi^*\rho}=\phi^{-1}\circ f_{\sigma,\rho}\circ \varphi, \quad E_{\varphi^*\sigma, \phi^*\rho}(f_{\varphi^*\sigma,\phi^*\rho}) = E_{\sigma, \rho} (f).
\end{equation}
Hence there is a well defined function $\mathcal E:\T(M)\times \T(M)\to \mathbb{R}$ given by
\[\mathcal E([\sigma]_{\Teich}, [\rho]_{\Teich}) := E_{\sigma, \rho}(f_{\sigma,\rho}),\]
where $\sigma$ and $\rho$ are arbitrary representatives of $[\sigma]_{\Teich}$ and $[\rho]_{\Teich}$, respectively.

Following the previous discussion we may introduce
\begin{equation}\label{eq:wolf}
\Phi_{[\sigma]_{\conf}}: \T(M) \to  \text{QD}([\sigma]_{\conf}),\;\;\; [\rho]_{\Teich} \mapsto (f_{\sigma, \rho}^*\rho)^{2,0},
\end{equation}
where this descends to Teichm\"uller space thanks \eqref{eq:equivariance}, i.e. using (for $\varphi = \Id$)
\begin{equation}\label{eq:wolf-equivariance}
	 \Phi_{\varphi^*\sigma}(\phi^*\rho) = \varphi^* \Phi_\sigma(\rho),\quad \varphi, \phi \in \mathrm{Diff}_0(M).
\end{equation}
The map $\Phi_{[\sigma]_{\conf}}$ is known to be a \emph{homeomorphism} by the results in \cite{Samp78,Wolf89}, see \cite[Theorem 3.1]{Wolf89}. (This gives in particular that $\T(M)$ is homeomorphic to $\mathbb{R}^{6g(M)-6}$ where $g(M)$ is the genus of $M$, since $6g(M) - 6$ is the real dimension of the space of holomorphic quadratic differentials.) We can think of $\Phi^{-1}_{[\sigma]_{\conf}}$ as a kind of exponential map for Teichm\"uller space; given $A\in  \text{QD}([\sigma]_{\conf})$ an explicit hyperbolic metric representing $\Phi_{[\sigma]_{\conf}}^{-1}(A)\in \T(M)$ is given in Lemma \ref{lemma:fg1} below in terms of $A$ and the Blaschke metric in the class $[\sigma]_{\conf}$.

\section{The bijection $\mc{G}$}\label{section:G}

As described in the Introduction, the map $\Phi_{[\sigma]_{\conf}}$ form \eqref{eq:wolf} can be used to define a map
 \[\mc{G}:\mathcal{T}\mathrm{QD}(M)\to \mathcal T(M)\times \mathcal T(M)\]
by setting
\[\mc{G}([\sigma]_{\conf}, A)=(\Phi^{-1}_{[\sigma]_{\conf}}(A),\Phi^{-1}_{[\sigma]_{\conf}}(-A)).\]
By \eqref{eq:wolf-equivariance}, $\mc{G}$ is invariant under the $\text{Diff}_{0}(M)$-action and hence it descends to a map on $\mathcal{T}\mathrm{QD}(M)$ (still denoted by $\mc{G}$).
The next proposition paves the way for the proof of Theorem \ref{thm:main}; we include a proof for completeness sake.

\begin{Proposition} The map $\mc{G}$ is a bijection. \label{prop:Ghomeo}
\end{Proposition}

\begin{proof} We first show that $\mc{G}$ is surjective. Let $\sigma$ and $g_0$ be hyperbolic metrics. We want to think of $g_0$ as fixed and $\sigma$ as a variable and
we let $\mathcal E^{[g_{0}]_{\Teich}}([\sigma]_{\Teich}) := \mathcal E([\sigma]_{\Teich}, [g_{0}]_{\Teich})$, so that we get a map
\[\mathcal E^{[g_{0}]_{\Teich}}: \mathcal T(M)\to \mathbb{R}_{>0}.\]
This map has some good properties as shown in \cite[Theorem 2.4]{FT87} and \cite{SY_79} (see also references therein): it is smooth, proper and its derivative at $[\sigma]_{\Teich}$ is given by
\begin{equation}
D\mathcal E^{[g_{0}]_{\Teich}}([\sigma]_{\Teich})(\eta)=-\langle \xi,\eta\rangle_{[\sigma]_{\Teich}},
\label{eq:derivative}
\end{equation}
where $\eta\in T_{[\sigma]_{\Teich}}\mathcal T(M)$ (identified with $\mathrm{QD}([\sigma]_{\conf})$), and $\xi=\mathrm{Re} (\varphi)$ with $\varphi = \Phi_{[\sigma]_{\conf}}([g_0]_{\Teich})$ is the Hopf differential of the unique harmonic map defined by $\sigma$ and $g_{0}$. The inner product $\langle{\bullet, \bullet}\rangle$ is given by the Weil-Petersson metric ($L^2$).

Fix a pair $[g_{1}]_{\Teich}, [g_{2}]_{\Teich}\in \mathcal T(M)$ and define the function $F:\mathcal T(M)\to \mathbb{R}_{>0}$
\[F([\sigma]_{\Teich}):=\mathcal E^{[g_{1}]_{\Teich}}([\sigma]_{\Teich}) + \mathcal E^{[g_{2}]_{\Teich}}([\sigma]_{\Teich}).\]
Note that $F$ is also proper. Using \eqref{eq:derivative} we see that $[\sigma]_{\Teich}$ is a critical point of $F$ if and only if
\[\Phi_{[\sigma]_{\conf}}([g_{1}]_{\Teich})+\Phi_{[\sigma]_{\conf}}([g_{2}]_{\Teich})=0.\]
Now since $F$ is proper and bounded below it has an absolute minimum $[\sigma_{\text{min}}]_{\Teich} \in \mc{T}(M)$. If we let $A:=\Phi_{[\sigma_{\text{min}}]_{\conf}}([g_{1}]_{\Teich})$, then we see that
\[\mc{G}([\sigma_{\text{min}}]_{\conf}, A) = ([g_{1}]_{\Teich}, [g_{2}]_{\Teich})\]
as desired. Note that the argument above shows (by definition of $\mc{G}$) that $\mc{G}([\sigma]_{\conf}, A) = ([g_{1}]_{\Teich}, [g_{2}]_{\Teich})$ if and only if $[\sigma]_{\Teich}$ is a critical point of $F$ and $A=\Phi_{[\sigma]_{\conf}}([g_{1}]_{\Teich}) = -\Phi_{[\sigma]_{\conf}}([g_{2}]_{\Teich})$.

To show that $\mc{G}$ is injective it suffices to invoke the proof of \cite[Proposition 2.12]{Schoen93} which shows that the global minimum $[\sigma_{\text{min}}]_{\Teich}$ is in fact the {\it unique}
critical point of $F$.
\end{proof}

\begin{Remark} It is natural to wonder if there is a similar uniqueness statement when considering a functional on Teichm\"uller space obtained by adding more than
two Dirichlet functionals. It has been recently shown that this fails as soon as we add three Dirichlet energies \cite{Mark_22}.

\end{Remark}

\section{The conjugacy $\mathcal{I}$ and proof of Theorem \ref{thm:main}}\label{section:I}

Let $(M, [\sigma]_{\conf})$ be a Riemann surface of genus $\geq 2$ and $g_{1}$ a hyperbolic metric. As before, we let $f:(M, [\sigma]_{\conf})\to (M, g_{1})$ be the unique harmonic diffeomorphism
isotopic to the identity and $A:=(f^{*}g_{1})^{2,0}$ the holomorphic Hopf differential associated to $f$. Recall that (see \eqref{eq:coupled-vortex-equations}) we denote by $g$ the Blaschke metric determined by the data $([\sigma]_{\conf}, A)$ so that $g$ and $A$ are linked by $K_{g}=-1+|A|^{2}_{g}$. We start with the following observation which is well known to experts, compare with \cite[Theorem 11.2]{Hi_87}; we include a proof for completeness.

\begin{Lemma} We have
\[f^{*}g_{1}=2\text{\rm Re}\, A+\left(1+|A|^{2}_{g}\right)g.\]
\label{lemma:fg1}
\end{Lemma}

\begin{proof} By definition of $A$, it follows that $2\text{\rm Re}\,A$ is the trace free part of $f^*g_{1}$, so the main claim in the lemma is that
\begin{equation}
\frac{1}{2}\text{Tr}_{g}(f^{*}g_{1}) = 1+|A|^{2}_{g}
\label{eq:trace}
\end{equation}
and this will use the fact that $f$ is harmonic and that $g$ is the Blaschke metric. To check this, it is convenient to let $g_{0}$ be the unique hyperbolic metric
in the same conformal class as $g$, so that $g=e^{2u}g_{0}$, for $u$ a uniquely determined smooth function. Next choose local isothermal coordinates $z$ and $w$ so that
$g_{0}=\sigma_{0}|dz|^2$ and $g_{1}=\sigma_{1}|dw|^2$ (in the following computations, $z$ is chosen near an arbitrary point $p \in M$ and $w$ near $f(p)$). By the definition of trace and a straightforward computation
\[\frac{1}{2}\text{Tr}_{g}(f^{*}g_{1})=e^{-2u}\left(\frac{f^*\sigma_{1}|f_{z}|^2}{\sigma_{0}}+\frac{f^*\sigma_{1}|f_{\bar{z}}|^2}{\sigma_{0}} \right).\]
If we let 
\[\mathcal H:=\frac{f^*\sigma_{1}|f_{z}|^2}{\sigma_{0}},\;\;\;\mathcal L:=\frac{f^*\sigma_{1}|f_{\bar{z}}|^2}{\sigma_{0}}\]
then \cite[(III) p. 453]{Wolf89} gives 
\begin{equation}\label{eq:auxiliary}
	\mathcal H\,\mathcal L=|A|^{2}_{g_{0}}
\end{equation}
and thus
\[\frac{1}{2}\text{Tr}_{g}(f^{*}g_{1})=e^{-2u}\left(\mathcal H+\frac{|A|^2_{g_{0}}}{\mathcal H} \right).\]
Since $|A|^2_{g}:=e^{-4u}|A|^2_{g_{0}}$, if we show that $\mathcal H=e^{2u}$, then \eqref{eq:trace} follows. But the equality of both functions follows from the fact that
 \cite[(VI) p. 453]{Wolf89} gives
 \[\Delta_{g_{0}}\log\mathcal H =2\mathcal H-2\mathcal L-2 = 2\mc{H} - \frac{2 |A|_{g_0}^2}{\mc{H}} - 2,\]
 where we also used \eqref{eq:auxiliary} in the second equality, and $g$ being the Blaschke metric gives (see \eqref{eq:coupled-vortex-equations})
 \begin{equation}
 \Delta_{g_{0}}u=e^{2u}-e^{-2u}|A|_{g_{0}}^{2}-1,
 \label{eq:B}
 \end{equation}
 where we also used the well-known formula for Gaussian curvature of $g = e^{2u} g_0$ (and that $g_0$ is hyperbolic)
 \begin{equation}\label{eq:Gauss-curvature}
 	K_{g}=e^{-2u}(-1-\Delta_{g_{0}}u).
 \end{equation}
Thus $2u=\log\mathcal H$ since there is a unique solution to the PDE \eqref{eq:B} (see \cite[Remark 5.3]{Mettler-Paternain-19}). Using again \eqref{eq:Gauss-curvature}, the PDE \eqref{eq:B} is equivalent to $K_{g}=-1+|A|_{g}^{2}$.
\end{proof}

\begin{Remark} It is possible to recast the bijection $\mathcal G$ and Lemma \ref{lemma:fg1} using AdS geometry; this originates in the work of Mess in \cite{Mess_07}. A pair $([g_{1}]_{\Teich}, [g_{2}]_{\Teich})$ uniquely determines a maximal global hyperbolic AdS structure on $M\times \mathbb{R}$.
It turns out that such an AdS manifold contains a unique Cauchy surface that is {\it maximal}, that is, it has zero mean curvature \cite{BBZ_07}. In addition its first fundamental form matches precisely the Blaschke metric $g$; one may see this by checking that the coupled vortex equations agree with the Gauss-Codazzi equations for the maximal surface. Moreover, the maximal surface comes equipped with the two ``projection" maps that are the unique harmonic diffeomorphisms isotopic to the identity that we described earlier having opposite Hopf differentials. This approach also delivers Lemma \ref{lemma:fg1} as explained in \cite[Section 7]{BS_20}. The bijection $\mathcal G$ may also be constructed using AdS geometry, see \cite[Theorem 3.8]{KS_07}  and \cite{BS_10}.

\label{remark:ads}

\end{Remark}

Since $\RE\, A$ is a real trace-free symmetric $2$-tensor on $M$ we can define for each $x\in M$ a $g$-symmetric trace-free linear map $\A_{x}:T_{x}M\to T_{x}M$ by the identity
\[g_{x}(\A_{x}(\bullet),\bullet)=\RE\, A_{x}(\bullet,\bullet).\]
Observe that $\A_{x}^2=-(\det \A_{x})\,\text{id}$ and it is straightforward to check that 
\begin{equation}\label{eq:detA}
\det\A_{x}=-\frac{1}{2} |\mathrm{Re}(A_x)|_g^2 = -|A_{x}|^{2}_{g}.
\end{equation}
The next lemma is essentially a consequence of Lemma \ref{lemma:fg1}. For brevity, let us set $h:=f^{*}g_{1}$.

\begin{Lemma} For each $x\in M$, the map $\text{\rm id}+\A_{x}:(T_{x}M,h_{x})\to (T_{x}M,g_{x})$ is an orientation preserving linear isometry.
\end{Lemma}

\begin{proof} Take $v\in T_{x}M$ and let us compute
\begin{align*}
g_{x}(v+\A_{x}(v),v+\A_{x}(v))=&g_{x}(v,v)+2g_{x}(v,\A_{x}(v))+g_{x}(\A_{x}(v),\A_{x}(v))\\
=&g_{x}(v,v)+ 2\RE \,A_{x}(v,v)+g_{x}(v,\A^{2}_{x}(v))\\
=&g_{x}(v,v)+ 2\RE \,A_{x}(v,v)+|A_{x}|^{2}_{g}\,g_{x}(v,v)\\
=&h_{x}(v,v),
\end{align*}
where in the second to last line we used \eqref{eq:detA}, and in the last line we used the definition of $h$ and Lemma \ref{lemma:fg1}. 

Since $\text{id}+\A_{x}$ is a linear isometry, $\pm 1$ cannot be eigenvalues of $\A_{x}$ (recall that $\A_{x}$ has zero trace). This implies right away
that $K_{g}<0$ or equivalently $|A|_{g}<1$ (thus obtaining a proof of \eqref{eq:curvature-condition-coupled-vortex} for $m=2$). Indeed, if there was a point in $M$ where $|A|_{g} > 1$, since $A$ has $4g(M) - 4 > 0$ zeroes with multiplicity (here $g(M)$ is the genus of $M$), by continuity we can also find a point where $|A|_g = 1$; this is however handled by the previous argument. Hence
the eigenvalues of $\text{\rm id}+\A_{x}$ are positive thus showing that it preserves orientation.
\end{proof}

\begin{Lemma}\label{lemma:inverse} The inverse of $\text{\rm id}+\A_{x}$ is the map
\[\mc{I}_{x}:=-K_{g}(x)^{-1}(\text{\rm id}-\A_{x}).\]
\end{Lemma}

\begin{proof} This is immediate from 
$$(\text{\rm id}+\A_{x})(\text{\rm id}-\A_{x})=\text{id}-\A_{x}^{2}=(1-|A_{x}|_{g}^2)\,\text{id}=-K_{g}(x)\,\text{id}.$$

\end{proof}

If we now let $x$ vary in $M$, the lemmas give a principal bundle isomorphism
\[\mc{I}:SM_{g}\to SM_{h}\]
where $SM_{g}$ and $SM_h$ denote the unit circle bundles of $g$ and $h$, respectively. This map will provide the conjugacy claimed in Theorem \ref{thm:main}.
Let $\{\alpha_{1},\beta_{1},\gamma_{1}\}$ (resp. $\{X_1, H_1, V_1\}$) and $\{\alpha, \beta, \psi\}$ (resp. $\{X, H, V\}$) denote the coframes (resp. frames) of 1-forms in $SM_{h}$ and $SM_g = SM$, respectively, constructed in \S \ref{ssec:geometry}. Write $\pi: SM \to M$ for the projection and recall that $\lambda=\text{Im}(\pi^*_{2}A)$.

\begin{Lemma}\label{lemma:computation} We have:
\begin{align*}
		\mc{I}^*\alpha_1 &= \left(1 + \frac{1}{2} V\lambda\right)\alpha - \lambda \beta,\\
		\mc{I}^*\beta_1 &= - \lambda \alpha + \left(1 - \frac{1}{2} V\lambda\right)\beta,\\
		\mc{I}^* \psi_1 &= \psi.
	\end{align*}
	\label{lemma:key}
\end{Lemma}	
	
\begin{proof} Let us compute for $(x,v)\in SM$ and $\xi\in T_{(x,v)}SM$:
	\begin{align*}
		(\mc{I}^*\alpha_1)_{(x, v)}(\xi) &= (\alpha_{1})_{(x, \mc{I}_{x}(v))} (d\mc{I}(\xi)) = h_{x}(\mc{I}_{x}(v), d\pi (\xi))\\
		&= g_{x}(v, \mc{I}_{x}^{-1} d\pi(\xi))\\
		&=g_x\left(v, d\pi(\xi) + \A_{x}( d\pi(\xi))\right)\\
		&= \alpha_{(x, v)}(\xi) + \RE\, A_{x}(v, d\pi(\xi)) = \alpha_{(x, v)}(\xi) + \RE\, \pi^* A_{x}(X, \xi),
	\end{align*}
	where in the second line we used that $\mc{I}_{x}$ is an isometry, in the third line Lemma \ref{lemma:inverse}, and in the last equality the definition of $\mathbb{A}_x$ and that $d\pi(X) = v$.  Using \eqref{eq:reA} yields
	
		\[
		\mc{I}^*\alpha_1 = \alpha + \iota_X \pi^*\RE\, A = \left(1 + \frac{1}{2} V\lambda\right) \alpha - \lambda \beta. 
	\]
	Similarly we compute, denoting $i$ and $i_1$ the rotations by $\frac{\pi}{2}$ of $g$ and $h$, respectively:
	\begin{align*}
		(\mc{I}^*\beta_1)_{(x, v)}(\xi) &= (\beta_1)_{(x, \mc{I}_{x}(v))} (d\mc{I}(\xi)) = h_{x}(i_1\mc{I}_{x}(v), d\pi(\xi))\\
		&= g_{x}(i v, \mc{I}_{x}^{-1} d\pi(\xi))\\
		&=g_{x}(i v, d\pi(\xi) + \A_{x}( d\pi(\xi)))\\
		&= \beta_{(x, v)}(\xi) + \RE\, A_{x}(iv, d\pi(\xi)) = \beta_{(x, v)}(\xi) + \RE\, \pi^* A_{x}(H, \xi).
	\end{align*}
	where again we used that $\mc{I}_x$ is an isometry, Lemma \ref{lemma:inverse}, as well as $\mc{I} i = i_1 \mc{I}$ since $\mc{I}$ commutes with rotations, and that $d\pi(H) = iv$. Thus using \eqref{eq:reA} again:
	\[
		\mc{I}^*\beta_1 = -\lambda \alpha + \left(1 - \frac{1}{2} V\lambda\right) \beta.
	\]
	
	Last we turn to $\mc{I}^*\psi_1$. As $\mc{I}$ commutes with the rotations, we have $\mc{I}^*V_1 = V$, so we may write
	\[\mc{I}^*\psi_1 = \psi + a \alpha + b \beta\]
	 for some functions $a$ and $b$. We will show that $a=b=0$ thus completing the proof. By the structure equations \eqref{eq:surface-geometry2}:
	\begin{align*}
		&\mc{I}^* d\alpha_1 = d \left(\left(1 + \frac{1}{2}V\lambda\right) \alpha - \lambda \beta\right)\\
		&=\mc{I}^* (\psi_1 \wedge \beta_1) = (\psi + a\alpha + b \beta) \wedge \left(-\lambda \alpha + \left(1 - \frac{1}{2}V\lambda\right) \beta\right).
	\end{align*}
	Equating the coefficients of $\alpha \wedge \beta$ we get
	\[
		0 = \left(-H\left(1 + \frac{1}{2}V\lambda\right) - X\lambda\right) \alpha \wedge \beta = \left(a\left(1 - \frac{1}{2}V\lambda\right) + \lambda b\right) \alpha \wedge \beta,
	\]
	where the first equality holds by \eqref{eq:Ahol}. Similarly,
	\begin{align*}
		&\mc{I}^* d\beta_1 = d \left(-\lambda \alpha + \left(1 - \frac{1}{2}V\lambda\right) \beta\right)\\
		& =- \mc{I}^*(\psi_1 \wedge \alpha_1) = (\psi + a\alpha + b \beta) \wedge \left(\left(1 + \frac{1}{2}V\lambda\right) \alpha -\lambda \beta\right).
	\end{align*}
	Equating the coefficients of $\alpha \wedge \beta$ we get
	\[
			0 = \left(H\lambda + X\left(1 - \frac{1}{2}V\lambda\right)\right) \alpha \wedge \beta = \left(a\lambda + b \left(1 + \frac{1}{2}V\lambda\right)\right) \alpha \wedge \beta,
	\]
	where again the first equality holds by \eqref{eq:Ahol}, and we used \eqref{eq:surface-geometry2}. This implies
	\[
		0 = \begin{pmatrix}
			1 - \frac{1}{2}V\lambda & \lambda\\
			\lambda & 1 + \frac{1}{2}V\lambda
		\end{pmatrix} \begin{pmatrix}
		a\\
		b
		\end{pmatrix}
	\]
	The determinant of the preceding matrix is equal to $1 - \frac{1}{4}(V\lambda)^2 - \lambda^2 = -K_g > 0$ where we used \eqref{eq:coupled-vortex-equations}, \eqref{eq:curvature-condition-coupled-vortex}, and \eqref{eq:norm}, so $a = b = 0$ as claimed.
\end{proof}

\subsection{Proof of Theorem \ref{thm:main}} We are now ready to put the pieces together and complete the proof of our main result. By Proposition \ref{prop:Ghomeo} the map $\mc{G}$ is a bijection. If $\mc{G}([\sigma]_{\conf}, A)=([g_{1}]_{\Teich}, [g_{2}]_{\Teich})$ by definition we know that $A$ is the Hopf differential of the unique harmonic diffeomorphism $f:(M, [\sigma]_{\conf})\to (M,g_{1})$ isotopic to the identity (and $-A$ is the Hopf differential of the unique harmonic diffeomorphism from $(M,[\sigma]_{\conf})$ to $(M,g_{2})$).
Obviously $h=f^*g_{1}$ is isometric to $g_1$ and its geodesic flow has weak stable foliation conjugate to that of $g_{1}$ via the lift of $f$. Next
we claim that the principal bundle isomorphism $\mc{I}:SM_{g}\to SM_{h}$ conjugates the weak foliation of $F=X+\lambda V$ with the weak foliation of $X_{1}$, the geodesic vector field
of $h$. Indeed Lemma \ref{lemma:key} implies

	\begin{align*}
		\mc{I}^*(\beta_1 + \psi_1) &= -\lambda \alpha + \left(1 - \frac{1}{2}V\lambda\right) \beta + \psi\\
		&=-\lambda\alpha-r^s\beta+\psi,
	\end{align*}
	where in the second line we used the notation of \eqref{eq:rus}, and hence (see \eqref{eq:weakbundles})
	\[\ker \mc{I}^*(\beta_1 + \psi_1)=\mathbb{R}F\oplus \mathbb{R}(H+r^{s}V).\]
	But for the hyperbolic metric $h$, the weak stable bundle is given by $\ker (\beta_{1}+\psi_{1})$ and the claim follows; the constructed conjugacy is explicitly $df \circ \mc{I}: SM_g \to SM_{g_1}$. Running the same construction with $-A$ and $g_2$
	(which amounts to switching $\lambda$ by $-\lambda$) we obtain a conjugacy between weak unstable foliations thus proving Items a, b, and c in Theorem \ref{thm:main}.
	The proof of Item d is a direct consequence of \cite[Theorem 5.5]{Mettler-Paternain-19}. Finally, Item e is proved in Section \ref{section:mls} below.
		\qed

	\begin{Remark} The explicit map $\mc{I}$ provides an alternative and simpler way to compute the resonant 1-forms at zero for $F$; these were originally computed in \cite[Theorem 1.5 and Section 8]{CP22} using vertical Fourier analysis. Indeed, if we let $\text{Res}_{0}^{1}(g_{2})$ denote the space of resonant 1-forms of $g_{2}$ at zero in the kernel of contraction by the geodesic vector field (a well-known object), then the analogous space of resonant $1$-forms $\mathrm{Res}_0^1(F)$ corresponding to $F$ is given by pull-back by the map $\mc{I}$ corresponding to $-A$:
	\[\text{Res}_{0}^{1}(F)=\mc{I}^*\text{Res}_{0}^{1}(g_{2}).\]
	Similarly for the co-resonant 1-forms of $F$ and the metric $g_{1}$. This gives right away that the winding cycles of the two SRB measures of $F$ vanish (something that is not obvious a priori) and that the dimension of resonant (and co-resonant) 1-forms is given by the first Betti number of $M$.

	\end{Remark}

\section{The marked length spectrum of the pair $([g_{1}]_{\Teich}, [g_{2}]_{\Teich})$.}\label{section:mls}
In this final section we use the conjugacy $\mc{I}$ constructed in the preceding section to prove \eqref{eq:mls1} and \eqref{eq:mls2} and we provide some final remarks. We recall that $\mc{G}([\sigma]_{\conf}, A) = ([g_1]_{\Teich}, [g_2]_{\Teich})$, which means that there is a unique harmonic diffeomorphism $f: (M, [\sigma]_{\conf}) \to (M, g_1)$ isotopic to the identity such that $(f^*g_1)^{2, 0} = A$; write $h = f^*g_1$.

\subsection{Proof of \eqref{eq:mls1} and \eqref{eq:mls2}}
Fix a non-trivial free homotopy class $\mathfrak{c}$ in $M$. To prove \eqref{eq:mls1} and \eqref{eq:mls2}, it suffices to consider a primitive class $\mathfrak{c}$ which we assume from now on (primitive here means that $\mathfrak{c}$ is not a power of order two or higher of another free homotopy class). Let $\tau$ and $\zeta$ denote the unique closed orbits of the geodesic flow of $h$ and of the flow generated by $F$, respectively, whose projections to $M$ belong to $\mathfrak{c}$. Indeed, the existence of unique closed geodesics of $(M, h)$ in the class $\mathfrak{c}$ is classical, see for instance \cite[Theorem 3.8.14]{Klingenberg-95}; the existence of the unique periodic orbit $\zeta$ can be justified as follows.  We regard quasi-Fuchsian flows as flows on the unit tangent bundle $T_{1}M$ of positive half-lines. Since they are parametrized by $\T(M)\times \T(M)$ which is a connected set, structural stability of Anosov flows implies that any quasi-Fuchsian is topological orbit equivalent to the geodesic flow of a hyperbolic metric $g_{0}$ via a H\"older homeomorphism isotopic to the identity. By existence and uniqueness of closed geodesics with respect $g_0$ in each free homotopy class in $M$, we then get the desired property on $SM_g$ and the thermostat flow.

Let $W^{ws}(\tau)$ and $W^{ws}(\zeta)$ be the corresponding weak stable leaves of $\tau$ and $\zeta$. Using that the strong stable leaves are immersed copies of $\mathbb{R}$, as well as the fact that the stable bundle of quasi-Fuchsian flows is orientable (see \cite[Section 8.3]{CP22} for a formula for the vector field spanning the stable bundle) we have the dichotomy: a weak stable leaf of $F$ is either an immersed cylinder (i.e. a copy of $\mathbb{S}^1 \times \mathbb{R}$) or an immersed plane (i.e. a copy of $\mathbb{R}^2$). The former possibility is valid if and only if the weak stable leaf contains a closed orbit.

We claim that $\mc{I}$ maps $W^{ws}(\zeta)$ (diffeomorphically) to $W^{ws}(\tau)$. Indeed, $\mc{I}(W^{ws}(\zeta))$ is a weak stable leaf which is an immersed cylinder so it contains a closed orbit $\tau'$ of the geodesic flow of $h$. Since $\tau'$ and $\mc{I}(\zeta)$ generate the fundamental group of $\mc{I}(W^{ws}(\zeta))$, they are homotopic up to orientation through curves in $\mc{I}(W^{ws}(\zeta))$. However, by Lemma \ref{lemma:computation} we get 
\begin{align*}
	\mc{I}_*(F) &= \mc{I}_*\left(1 + \frac{1}{2}V\lambda\right) X_1 - \mc{I}_*\lambda (H_1 - V_1),\\
	\mc{I}_*(H + r^sV) &= - \mc{I}_*\lambda\, X_1 - \mc{I}_*\left(-1 + \frac{1}{2} V\lambda\right) (H_1 - V_1).
\end{align*}
It is easily computed that the Jacobian of $\mc{I}: W^{ws}(\zeta) \to W^{ws}(\tau)$ (where the two weak stable leaves are equipped with the orientations induced by the natural frames) is equal to $-K_g > 0$, so $\mc{I}$ is orientation preserving. Hence we get that in fact $\tau'$ and $\mc{I}(\zeta)$ are homotopic as curves with the same orientation. Since $\mc{I}(\zeta)$ has the same projection as $\zeta$ to $M$ (recall $\mc{I}$ covers the identity), we conclude that projection of $\tau'$ belongs to the class $\mathfrak{c}$. Thus by uniqueness of closed orbits we get $\tau' = \tau$, and the claim is proved.

Since $\iota_{X_{1}}d\alpha_{1}=0$, it follows that $\alpha_{1}|_{W^{ws}(\tau)}$ is a closed $1$-form in $W^{ws}(\tau)$. Thus $\mc{I}^*\alpha_{1}|_{W^{ws}(\zeta)}$ is a closed $1$-form in $W^{ws}(\zeta)$ and using that $\mc{I}(\zeta)$ and $\tau$ are homotopic through curves in $W^{ws}(\tau)$ we have
 \[\int_{\zeta}\mc{I}^{*}\alpha_{1} = \int_{\mc{I}(\zeta)} \alpha_1 = \int_{\tau}\alpha_{1}.\]
The right hand side is simply $\ell_h(\mathfrak{c}) = \ell_{g_{1}}(\mathfrak{c})$ (here we use that the harmonic map $f$ is isotopic to the identity) and the left hand side can be computed using Lemma \ref{lemma:key} to obtain
\[\mc{I}^{*}\alpha_{1}(F)=1+\frac{V\lambda}{2}.\]
Thus
\[\int_{\zeta}\mc{I}^*\alpha_{1}=\ell_{g}(\pi \circ \zeta)+\frac{1}{2}\int_{\zeta}V\lambda=\ell_{g_{1}}(\mathfrak{c})\]
which proves \eqref{eq:mls1}. The proof of \eqref{eq:mls2} is analogous.

\subsection{Equality of marked length spectra and Proof of \eqref{eq:arithmetic-mean}} Observe that equations  \eqref{eq:mls1} and \eqref{eq:mls2} imply that the thermostat flows associated with $(g,A)$ and $(g,-A)$ have the same marked length spectrum in the sense that the periods of closed orbits in a given free homotopy class match. More precisely, let $\mathfrak{c}$ be a free homotopy class on $M$, and let $\zeta_{\pm}$ be periodic orbits of the flow generating by $F_\pm := X \pm \lambda V$ such that $\pi \circ \zeta_{\pm}$ belongs to the class $\mathfrak{c}$; here $\lambda = \im(\pi_2^*A)$. Using that 
\[
	\mc{G}([\sigma]_{\conf}, A) = ([g_1]_{\Teich}, [g_2]_{\Teich}), \quad \mc{G}([\sigma]_{\conf}, -A) = ([g_2]_{\Teich}, [g_1]_{\Teich}),
\]
equations \eqref{eq:mls1} and \eqref{eq:mls2} imply
\begin{align*}
	\ell_g(\pi \circ \zeta_{+}) &= \ell_{g_1}(\mathfrak{c}) - \frac{1}{2} \int_{\zeta_+} V\lambda,\quad &\ell_g(\pi \circ \zeta_+) = \ell_{g_2}(\mathfrak{c}) + \frac{1}{2} \int_{\zeta_+} V\lambda,\\
	\ell_g(\pi \circ \zeta_{-}) &= \ell_{g_2}(\mathfrak{c}) + \frac{1}{2} \int_{\zeta_-} V\lambda,\quad &\ell_g(\pi \circ \zeta_-) = \ell_{g_1}(\mathfrak{c}) - \frac{1}{2} \int_{\zeta_-} V\lambda,
\end{align*}
Manipulating the first and the second equality, as well as the third and the fourth one, we get
\begin{equation}\label{eq:flip}
	\int_{\zeta_+} V\lambda = \ell_{g_1}(\mathfrak{c}) - \ell_{g_2}(\mathfrak{c}) = \int_{\zeta_-} V\lambda
\end{equation}
as well as
\begin{equation*}
	\ell_g(\pi \circ \zeta_+) = \frac{1}{2} (\ell_{g_1}(\mathfrak{c}) + \ell_{g_2}(\mathfrak{c})) = \ell_g(\pi \circ \zeta_-).
\end{equation*}
The latter equality says that the periods of $\zeta_\pm$ agree.

Since there is a H\"older orbit equivalence between the flows generated by $F_\pm$, a well-known application of the Liv\v{s}ic theorem gives that the flows are in fact H\"older conjugate (the conjugacy maps $\zeta_{+}$ to $\zeta_{-}$). This conjugacy {\it cannot} be upgraded to a smooth conjugacy as long as $A\neq 0$. Indeed, if we had a smooth conjugacy between the flows of $F_\pm$, then the weak stable foliations of the geodesic flows of $g_1$ and $g_2$ would be conjugate. In particular, the holonomies of the two foliations (i.e. the holonomy of foliated bundles, see \cite[Chapter 2.1]{CC_00}) are conjugate, and these holonomies can be shown to agree with the representations of $\pi_1(M) \to \mathrm{PSL}(2, \mathbb{R})$ given by the respective hyperbolic structures. This implies that $[g_{1}]_{\Teich} = [g_{2}]_{\Teich}$ and thus $A=0$ by Theorem \ref{thm:main}, Item d. The question of when a continuous conjugacy between dissipative Anosov flows on $3$-manifolds is smooth is analyzed in the recent preprint \cite{GLRH_24}. Note that the H\"older conjugacy for quasi-Fuchsian flows must in fact swap SRB measures. This is a consequence of the observation that $r^{u}(g,A)=-r^{s}(g,-A)$ together with \eqref{eq:flip}, the fact that $\zeta_{+}$ and $\zeta_{-}$ have the same period, and Bowen's formula. This phenomenon is an obstruction to smooth conjugacy of transitive Anosov flows as explained in \cite{GLRH_24}.

\bibliographystyle{alpha}
\bibliography{Biblio}

\end{document}